\documentclass[11pt,leqno]{amsart}
\usepackage{amsmath, amsthm, amssymb,xspace}
\usepackage[breaklinks=true]{hyperref}
\usepackage{amsmath,amscd}
\theoremstyle{plain}
\newtheorem{theorem}{Theorem}[section]
\newtheorem{lemma}{Lemma}[section]

\theoremstyle{definition}

\newtheorem*{acknowledgement}{Acknowledgements}

\newcommand{\bggo}{\mathcal O}

\newcommand{\mf}[1]{\displaystyle{\mathfrak{#1}}}

\newcommand{\comment}[1]{}

\DeclareMathOperator{\spec}{\ensuremath{Spec}}
\DeclareMathOperator{\Gr}{\ensuremath{gr}}

\DeclareMathOperator{\ad}{\ensuremath{ad}}
\DeclareMathOperator{\Id}{\ensuremath{Id}}

\DeclareMathOperator{\Sym}{\ensuremath{Sym}}
\DeclareMathOperator{\Ann}{\ensuremath{Ann}}

\begin{document}

\title[Quotients of infinitesimal Cherednik algebras]{On
maximal primitive quotients of infinitesimal Cherednik algebras of
$\mf{gl_n}$}

\author{Akaki Tikaradze}
\address{The University of Toledo, Department of Mathematics, Toledo, Ohio, USA}
\email{\tt tikar06@gmail.com}
\begin{abstract}
We prove analogues of some of Kostant's theorems for infinitesimal Cherednik
algebras of $\mf{gl}_n$. As a consequence, it follows that in positive
characteristic the Azumaya  and smooth loci of the center of these algebras coincide.
\end{abstract}
\maketitle

\section{Introduction}

Infinitesimal Cherednik algebras (more generally, infinitesimal Hecke
algebras) were introduced by Etingof, Gan and Ginzburg \cite{EGG}.
Here we will be concerned with infinitesimal Cherednik algebras of
$\mf{gl}_n$.
Let us recall the definition.
Let $\mf{h}=\mathbb{C}^n$ denote the standard representation of
$\mf{g}=\mf{gl}_n.$
Denote by $y_i$ the standard basis elements of $\mf{h},$ and by $x_i$ the
dual basis of $\mf{h}^*$. For the given deformation parameter
$b=b_0+b_1\tau+\cdots+b_m\tau^m \in \mathbb{C}[\tau], b_m\neq 0, m\geq 0$, one defines the
infinitesimal Cherednik algebra of $\mf{gl}_n$ with parameter $b$, to be
denoted by $H_b$, as the quotient of the semi-direct product $\mf{Ug}
\ltimes T(\mf{h} \oplus \mf{h^{*}})$ by the relations
$$[x, x']=0, \qquad [y, y']=0, \qquad [y, x] = b_0 r_0(x, y) + b_1 r_1(x,
y) + \cdots + b_mr_m(x, y),$$

\noindent where $x, x'\in \mf{h^*}, y, y'\in \mf{h}$, and $r_i(x, y)
\in \mf{Ug}$ are the symmetrizations of the following functions on
$\mf{g}$ (thought of as elements in $\Sym \mf{g}$ via the trace pairing):
$$(x, (1 - tA)^{-1}y) \det(1 - tA)^{-1} = r_0(x, y)(A) + r_1(x, y)(A)t +
r_2(x, y)(A)t^2 + \cdots$$

The algebras $H_b$ have the following PBW property. If we introduce the
filtration on $H_b$ by setting
$\deg x=\deg y=1, x\in \mf{h^*}, y\in \mf{h}, \deg g = 0, g\in \mf{g},$
then the natural map
$: \mf{U}\mf{g} \ltimes \Sym(\mf{h}\oplus\mf{h}^*)\to \Gr H_b$ is an
isomorphism.

The enveloping algebra $\mf{U}(\mf{sl}_{n+1})$ is an example of $H_b$
for $m=1$ (Example 4.7 \cite{EGG}). In fact, the algebras $H_b$ have many properties similar to
the enveloping algebras of simple Lie algebras. We will introduce a Poisson variety (for each $m:=\deg b$)
which can be thought of as an analogue of the nilpotent cone of a semi-simple Lie algebra.
Our first result shows  that it is an irreducible reduced normal variety (Theorem \ref{kostant}), 
an analogue of Kostant's classical result.
As an application,  we will describe annihilators of Verma modules of $H_b,$ and
show that in positive characteristic the Azumaya locus of $H_b$ coincides with the smooth
locus of its center.

\section{The main results}
Since $H_b\simeq H_{ab}$ for any $a\in \mathbb{C}^{*}.$ we will assume from now on
that $b$ is monic $:b_m=1.$

Besides the natural action of $G = GL_n(\mathbb{C})$ on $H_b$, we also have the 
action of $\mf{h}$ and $\mf{h^{*}}$ defined as follows. For any $v\in
\mf{h},$ the adjoint action $\ad(v)$ is locally nilpotent on $H_b$. Thus
$exp(\ad(v))$ gives an automorphism of $H_b$, and in this way $\mf{h}$
acts on $H_b$. The action of $\mf{h}^*$ on $H_b$ is defined similarly.
Combining these actions with the $G$-action, we get the actions of $G
\ltimes \mf{h}, G \ltimes \mf{h}^*$ on $H_b$.

Let $Q_1, \cdots, Q_n\in k[\mf{g}]^G$ be defined as follows:
$$ \det(t\Id-X)=\sum_{j=0}^n (-1)^j t^{n-j}Q_j(X).$$

Also let $\alpha_1,\cdots, \alpha_n$ be the corresponding elements of
$\mathbf{Z}(\mf{Ug})$ under the symmetrization identification of
$\mathbb{C}[\mf{g}]^G$ and $\mathbf{Z}(\mf{Ug})$.
It was shown in \cite{T1} that the following elements generate the center of $H_b$
$$t_i=\sum_j[\alpha_i, y_j]x_j-c_i=\sum_j y_j[x_j,\alpha_i]-c_i\in \mathbf{Z}(H_b),$$
where
$c_i\in \mathbf{Z}(\mf{Ug})$ are certain elements.  The top
symbols of $c_i$ are given as follows. Let us consider
the following element of $\mathbb{C}[\mf{g}][t, \tau]$ given by
$$c' = \frac{\det(t-A)}{(t\tau-1)\det(1 - \tau A)} $$
then the top symbol of $c_i$ considered as an element of
$\mathbb{C}[\mf{g}]$ is the coefficient of $t^{n-i} \tau^m$ in $c'$.
We have that
$\mathbf{Z}(H_b) = \mathbb{C}[t_1, \cdots, t_n].$ For a character
$\chi : \mathbb{C}[t_1,\cdots , t_n] \to \mathbb{C},$ denote by $U_{b,
\chi}$ the quotient $H_b/\ker(\chi)H_b.$

 From now on we will assume that $m \geq 1$.
 Let us introduce a  new filtration
on $H_b$ by setting
$\deg x_i=m, \deg y_i=1, \deg g=1, g\in \mf{g}$. Then, $\Gr H_b=\Sym
(\mf{g}\oplus\mf{h}\oplus \mf{h}^*)$ is a Poisson algebra (and the
Poisson bracket depends only on $m$). We will denote it by $A_m$.
Denote $B_m =\Gr H_b/(\Gr t_1,\cdots, \Gr t_n)$. Again, $B_m$ is a Poisson algebra.
 Variety $\spec B_1$ is the nilpotent cone of
$\mf{sl}_{n+1}(\mathbb{C})$. 
The main result of this paper is the following analogue of some of
 Kostant's theorems for semi-simple Lie algebras \cite{K}.

\begin{theorem}\label{kostant}
The algebra $H_b$ is a free module over its center.
$B_m$ is an integral domain, which is a normal, complete intersection ring.
Moreover, the smooth locus of $\spec B_m$ under the Poisson bracket is
symplectic.
\end{theorem}

\begin{proof}
We will partially follow \cite{BL}.
Denote by $f_x$ (resp. $f_y$) the element 
$\det(\{\alpha_i, x_j\})_{i, j}\in B_m$ (resp. $\det(\{\alpha_i, y_j\})_{i j}\in B_m$).
 Then the localization $(B_m)_{f_x}$ is isomorphic to the
localized polynomial algebra $\Sym (\mf{g}\oplus \mf{h})_{f_x}$. A similar statement
holds for $(B_m)_{f_y}.$ We will use the notation 
$D(f)=\spec (B_m)_f\subset \spec B_m, f\in B_m.$
Let us set $U = D(f_x) \cup D(f_y)$. To show that $X=\spec B_m$ is an irreducible, 
reduced and normal variety, it is enough to show that it is Cohen-Macaulay, 
$U$ is connected, and 
$\dim (X\setminus U) \leq \dim X-2$ [\cite{BL},  Corollary 2.3].

We have an action of the affine group $G\ltimes \mf{h}$ on $\Sym
(\mf{g}\oplus \mf{h})$.
Then $f_x$ is a semi-invariant of this action, i.e., $(g, v)f_x =
\det(g)f_x, g\in G, v\in \mf{h}$. As explained in [\cite{R1},  Theoreme 3.8], the set
$D(f_x)\subset \spec \Sym (\mf{g}\oplus \mf{h})$ is the dense orbit under
the action of $G\ltimes \mf{h}$  on $\spec \Sym (\mf{g}\oplus \mf{h})$. In
fact, this set consists of pairs $(A, v)$ with $A\in \mf{g}, v\in
\mf{h}$, such that $v, Av,\cdots, A^{n-1}v,$ are linearly independent.
We have a similar statement about $D(f_y)$, and the action of $G\ltimes
\mf{h}^*$ on  $\Sym (\mf{g}\oplus \mf{h}^*)$.

It was shown in \cite{T1} that the algebra $\mathbb{C}[\alpha_1,\cdots, \alpha_n]$ is finite
over $\mathbb{C}[c_1,\cdots, c_n].$ In particular,
$\mathbb{C}[c_1,\cdots, c_n]$ is isomorphic to the polynomial algebra in
$n$ variables. Hence,  $\mathbb{C}[\alpha_1,\cdots, \alpha_n]$ being 
a Cohen-Macaulay algebra, it must be a finitely generated
projective module over  $\mathbb{C}[c_1,\cdots, c_n].$
 Therefore, by the Quillen-Suslin theorem, $\mathbb{C}[\alpha_1,\cdots, \alpha_n]$ is a
finitely generated free module over $\mathbb{C}[c_1,\cdots, c_n]$.

Let us introduce a filtration on $A_m,$ where 
$\deg g=1, g\in \mf{g}, \deg x_i=\deg y_j=0$. Since $\Sym \mf{g}$ is a free 
$\mathbb{C}[\alpha_1,\cdots, \alpha_n]$-module (by Kostant's theorem for $\mf{g}$ \cite{BL}), we
conclude that $\Sym \mf{g}$ is a free
$\mathbb{C}[c_1,\cdots, c_n]$-module.
This implies that $(t_1,\cdots, t_n)$ is a regular sequence (since $\Gr c_j=\Gr t_j$)
 and $\Gr A_m$
is a free module over $\mathbb{C}[\Gr t_1,\cdots, \Gr t_n].$ Therefore
$A_m$ is a free module over $\mathbb{C}[\Gr t_1,\cdots, \Gr t_n],$ where $\Gr$
refers to the first filtration on $H_b.$
In particular, $H_b$ is a free $\mathbf{Z}(H_b)$-module.
Also, we obtain that $B_m$ is a complete intersection ring. In particular,
$X=\spec B_m$ is a Cohen-Macaulay variety.

Let us put $Y=X\setminus U.$
The latter filtration on $A_m$ induces the corresponding filtration on its quotient $B_m$.
We will denote the degeneration of $X$ (resp. $Y)$ under this filtration by $X'$ (resp. $Y').$
Then $X'$ is given by equations $c_i=0, i=1,\cdots, n.$ Similarly, $Y'$
 is given by
$f_x=f_y=0, c_i=0, i=1,\cdots, n$.
Therefore,  we get that $X'=\mf{h} \times \mf{h}^* \times N$ and 
$Y'=\mf{h} \times \mf{h}^* \times N\cap (f_x = 0 = f_y),$ where $N$ denotes
the nilpotent cone of $\mf{g}.$
We need to prove that $\dim Y\leq \dim X-2.$
Since $\dim Y=\dim Y',$ it is equivalent to showing  
that $\dim Y' \leq \dim X-2=\dim X'-2=\dim N + 2 \dim \mf{h} - 2$.
Consider the projection map $p : Y' \to N$. Let $W \subset N$ denote the open
subset of regular nilpotent matrices. Then, by the before mentioned result
of \cite{R1}, we have $\dim p^{-1}(W) \leq
\dim N + 2 \dim \mf{h} - 2$, and $p^{-1}(N\setminus W)=(N\setminus W) \times \mf{h}
\times\mf{h}^*$, whose dimension is $\dim N + 2 \dim \mf{h} - 2$. This proves the desired inequality.

Denote by $U'$ the smooth locus of $X.$ We have $U\subset U'.$
It is obvious that $D(f_x) \cap D(f_y)$ is nonempty. 
It is also clear that $D(f_x) \cup D(f_y)$ is in the orbit of any element of $D(f_x)
\cap D(f_y)$ under the actions of $G \ltimes \mf{h}, G \ltimes \mf{h}^*$. Since
both $G \ltimes \mf{h}, G \ltimes \mf{h}^*$ are connected algebraic groups preserving
the Poisson structure of $X,$
it follows that $U$ lies in a single symplectic leaf of $U'.$ Therefore, $U$ is a symplectic
variety and since its complement in $U'$  has the codimension $\geq 2,$
 it
follows that $U'$ is also symplectic.
\end{proof}

We will use the following standard simple 

\begin{lemma}\label{regular}

Let $A$ be a nonnegatively filtered $\bold{k}$-algebra (where $\bold{k}$ is a field) such that $\Gr A$ is commutative. 
Suppose that $z_1,\cdots, z_n\in\bold{Z}(A)$ are central elements such that $\Gr z_1,\cdots,\Gr z_n$ is
a regular sequence in $\Gr A.$ Then $\Gr (A/(z_1,\cdots, z_n))=\Gr A/(\Gr z_1,\cdots, \Gr z_n).$

\end{lemma}

\begin{proof}

We need to show that $\Gr(\sum_i x_iz_i)\in (\Gr z_1,\cdots, \Gr z_n)$ for all $x_i\in A.$
We may assume that $\sum \Gr x_i\Gr z_i=0.$ We proceed by the induction
on $\sum \deg(x_i).$ It follows from the regularity of the sequence $(\Gr z_1,\cdots, \Gr z_n)$
that there exist $a_1,\cdots,  a_n\in A,$ such that $\Gr a_i=\Gr x_i, 1\leq i\leq n$ and $\sum_i a_iz_i=0.$
Now replacing $x_i$ by $x_i-a_i$, we are done by the inductive assumption.

\end{proof}

As a consequence of the proof of  Theorem \ref{kostant} and Lemma \ref{regular}, 
we get that $\Gr U_{b,\chi}=B_m$ is a domain, so 
$U_{b,\chi}$ is also a domain.  

In analogy with semi-simple Lie algebras, one defines an analogue of the
category $\bggo$, and Verma modules for $H_b$ \cite{T1}. Let us recall their definition.
Denote by $n_{+}$ (resp. $n_{-}$) the Lie subalgebra of $\mf{g}$
consisting of upper (resp. lower) triangular matrices. Then we have a
triangular decomposition $H_b = H_- \otimes \mf{U}(C) \otimes H_+$, where
$H_+$ (resp. $H_-$) denotes the subalgebra of $H_b$ generated by
$n_{+}$ and $\mf{h}$ (resp. $n_{-}$ and $\mf{h}^*)$, and $C\subset \mf{g}$ is the Cartan
subalgebra of all diagonal matrices. Denote by $L_{+}$ (resp. $L_{-})$ the Lie algebra
$n_{+}\ltimes \mf{h}$ (resp. $n_{-}\ltimes \mf{h}^{*}).$ Thus, $H_{+}$ (resp. $H_{-})$
is the enveloping algebra $\mf{U}L_{+}$ (resp. $\mf{U}L_{-}).$

For a weight $\lambda \in C^*$, the
corresponding Verma module $M(\lambda)$ is defined as $H_b
\otimes_{\mf{U}(C)\otimes H_+} \mathbb{C}_\lambda$, where
$\mathbb{C}_{\lambda}$ is the 1-dimensional representation of $\mf{U}(C)
\otimes H_+$ on which $C$ acts by $\lambda$ and $L_{+}$ acts by 0.

The category $\bggo{}$ (analogue of the BGG category $\bggo{}$ of semi-simple Lie algebras)
is defined as the full subcategory of the category of finitely generated left $H_{b}$-modules
whose objects are modules on which $C$ acts semi-simply and $L_{+}$ acts locally nilpotently. 

We have the following analogue of a theorem of Duflo \cite{D}.

\begin{theorem}\label{Verma}
The annihilator of a Verma module $M(\lambda)$ is generated by
$\Ann(M(\lambda))\cap \mathbf{Z}(H_b)$.
\end{theorem}

\begin{proof}

The following lemma and its proof is  directly analogous to [\cite{J}, Corollary 2.8].
We present it for the completeness sake.
In what follows $GK$(-) denotes the Gelfand-Kirillov dimension.

\begin{lemma}  GK$(H_b/Ann(M(\lambda)))=2GK(M(\lambda)).$

\end{lemma}

\begin{proof} At first, we will show that $GK(H_b/Ann L(\lambda))=2GK(L(\lambda)),$
where $L(\lambda)$ is the simple module with the highest weight $\lambda.$ 
Since $L_{-}$ acts locally nilpotently on $H_{b},$ we get an imbedding 
$H_b/Ann L(\lambda)\to D(L_{-}, L(\lambda))$ where $D(L_{-}, L(\lambda))$ is
the subalgebra of $End_{\mathbb{C}}(L(\lambda))$ consisting of elements
annihilated by some power of $\ad (L_{-}).$ According to [\cite{J}, Lemma 2.6]\\
$GK(D(L_{-}, L(\lambda)))\leq 2GK(L(\lambda)),$ thus $GK(H_b/Ann L(\lambda))\leq 2GK(L(\lambda)).$ 
Let $v_{\lambda}$
be a maximal weight vector of $L(\lambda).$ Thus $L(\lambda)=H_{-}v_{\lambda}.$
 Let us choose $\delta\in C$
so that $\ad(\delta)$ has positive (negative) eigenvalues on $L_{+} (L_{-}).$ For $a\in\mathbb{C},$ 
and $H_b$-module
$M,$ we will denote by $M_a\subset M$ the space of eigenvectors of $\delta$ 
with the eigenvalue $a.$ In particular,
$L(\lambda)_{\lambda(\delta)}=\mathbb{C}v_{\lambda},$ and $L(\lambda)=\oplus_{l\geq 0} L(\lambda)_{\lambda(\delta)-l}.$
Then for any $a\in L(\lambda)_{\mu},$
there is an element $c\in H_{b},$ such that $ca=v_{\lambda}.$ But since $\lambda$ is the maximal
weight of $L(\lambda),$ using the triangular decomposition $H=H_{-}\otimes \mf{U}(C)\otimes H_{+}$ we may
choose $c\in H_{+}.$ Thus, for any $a\in L(\lambda)_{\mu}, b\in L(\lambda)_{\mu'}$ there exist
$\alpha\in {(H_{+})}_{\lambda(\delta)-\mu}, \alpha'\in {(H_{-})}_{\mu'-\lambda(\delta)}$ such that $b=\alpha'\alpha a.$
Let us denote by $\rho$ the quotient map $H_{b}\to H_{b}/Ann L(\lambda).$ Thus, 
$\dim Hom_{\mathbb{C}}(L(\lambda)_{\mu}, L(\lambda)_{\mu'})\leq \dim \rho((H_{-})_{\lambda(\delta)-\mu}(H_{+})_{\mu'-\lambda(\delta)}).$
Let $F_l=\sum_{i\leq l}(\mf{g}\oplus\mf{h}\oplus \mf{h}^{*})^i\subset H_b, l\geq 0.$ Then it follows that there is a positive integer
$k>0$ such that $F_lv_{\lambda}\subset \sum_{i\leq kl} L(\lambda)_{(\lambda(\delta)-i)}$ and
$(H_{-})_{-l}\subset F_{kl}\cap H_{-}, (H_{+})_{l}\subset F_{kl}\cap H_{+},$ for all $l>0.$
Now it follows that $\dim\rho(F_{2kl})\geq (\dim F_{l/k}v_{\lambda})^2.$
This implies that $GK(H_b/Ann L(\lambda))\geq 2GK(L(\lambda)),$ and so $GK(H_b/Ann L(\lambda))=2GK(L(\lambda)).$

  Suppose that $L(\lambda_i), i=1,\cdots, l$ are the elements of the Jordan-Holder series of $M(\lambda)$
  ($M(\lambda)$ has a finite length [\cite{T1},  Thm 4.1]). Then, 
  $$GK(H_b/Ann M(\lambda))=Max_i\{GK(H_b/Ann L(\lambda_i))\}=$$
  $$2Max_i\{GK(L(\lambda_i))\}=2GK(M(\lambda)).$$
\end{proof}

Now since $H_b/\Ann(M(\lambda))\cap\mathbf{Z}(H_b)$ is a domain and its quotient\\ 
$H_{b}/(Ann M(\lambda))$ has the same GK-dimension as \\
$2GK(M(\lambda))=GK(H_b/\Ann(M(\lambda))\cap\mathbf{Z}(H_b)),$
we conclude using [\cite{BK}, 3.5.] that
$$\Ann(M(\lambda)) = (\Ann(M(\lambda))\cap \mathbf{Z}(H_b))H_b.$$
\end{proof}

This implies that maximal primitive quotients of $H_b$ are precisely
algebras $U_{b, {\chi}}, \chi\in \spec \bold{Z}(H_b).$ Indeed, by \cite{T1}
every primitive quotient of $H_b$ has the form $H_b/\Ann L(\lambda), \lambda\in C^{*}.$
Let $M(\lambda')$ be an irreducible Verma module which belongs to the same block
of the BGG category $\bggo{}$ as $L(\lambda).$ Thus 
$\Ann(M(\lambda'))\cap\bold{Z}(H_b)=\Ann(L(\lambda))\cap\bold{Z}(H_b).$ Therefore, using Theorem
\ref{Verma}, we conclude that  $\Ann(M(\lambda'))\subset \Ann(L(\lambda)).$
Thus $H_b/\Ann(L(\lambda))$ is a quotient of $U_{b, \chi},$ where $\chi$ is the character of $\bold{Z}(H_b)$
corresponding to $M(\lambda').$

\section{The Azumaya locus}
Let us discuss the case of a field $\mathbf{k}=\bar{\bold{k}}$ of positive
characteristic. As before, let $b\in \bold{k}[\tau], \deg b=m>1,$ be
a monic polynomial. If $p$ is large enough (with respect to $m$) then the
definition of $H_b$ over $\mathbf{k}$ makes sense.
One checks easily that $\mf{h}^p, \mf{h^{*}}^p, g^p-g^{[p]}\in
\mathbf{Z}(H_b), g\in \mf{g},$ where $g^{[p]}\in \mf{g}$ denotes the
$p$-th power of $g$ as a matrix \cite{T1}. We will denote by $\mathbf{Z}_0(H_b)$
the algebra generated by the above elements. Also, for $p>> 0$ central
elements $t_1,\cdots, t_n\in \bold{Z}(H_b)$ are defined.

We have the following result which was conjectured in \cite{T1}.

\begin{theorem}\label{azumaya}
The smooth and Azumaya loci of $\mathbf{Z}(H_b)$ coincide, and
$\mathbf{Z}(H_b)$ is generated by $t_1,\cdots, t_n$ over
$\mathbf{Z}_0(H_b).$
The PI-degree of $H_b$ is $p^{\frac{1}{2}(n^2+n)}.$
\end{theorem}
\begin{proof}
Algebra $B_m$ can be defined over $\mathbb{Z}[\frac{1}{d!}]=R$ for large enough $d.$ 
Call this algebra $\tilde{B}_m.$ Thus, $B_m=\tilde{B}_m\otimes_R\mathbb{C}.$
Since by Theorem \ref{kostant} $\spec \tilde{B}_m\otimes_R\mathbb{C}$ is an irreducible normal Poisson
variety whose regular locus is symplectic,
it follows that for large enough $p=char \bold{k},$ a similar statement holds
for $\bar{B}_m=\tilde{B}_m\otimes_R\bold{k}.$ Since 
$\Gr H_b/(t_1-a_1,\cdots, t_n-a_n)=\bar{B}_m, a_1,\cdots, a_n\in \bold{k}$  (by \ref{regular}),
the claim now follows from  [\cite{T2}, Theorem 2.3] and the following simple lemma.

\end{proof}

\begin{lemma}
Let $S$ be an affine Poisson algebra over $\mathbf{k}$, and let $(f_1,\cdots,
f_n)$ be a regular sequence of Poisson central elements. Let
$S/(f_1,\cdots, f_n)$ be a normal domain such that its smooth locus is
symplectic. Then the Poisson center of $S$ is generated as an algebra by
$S^p, f_1, \cdots, f_n.$
\end{lemma}

\begin{proof}
Let us denote the ideal $(f_1,\cdots ,f_n)$ by $I.$
It follows immediately that the Poisson center of $S$ lies in $S^p+I$
[\cite{T2}, proof of Lemma 2.4]. Let $f \notin S^p[f_1,\cdots, f_n]$ be in the Poisson center of $S$. 
Then there is $k$ such that $f\in (S^p[f_1,\cdots ,f_n]+I^k)\setminus (S^p[f_1,\cdots, f_n]+I^{k+1}).$
Let us write $f=g+h,$  where $g\in S^p[f_1,\cdots, f_n], h\in I^k\setminus (S^p[f_1,\cdots, f_n]+I^{k+1}).$
But
$I^k/I^{k+1}$ is a free Poisson $S/I$-module.
Indeed, since $f_1,\cdots ,f_n$ 
is a regular sequence, it follows that $I^k/I^{k+1}$ is a free
$S/I$-module with the basis $f_1^{m_1}f_2^{m_2}\cdots f_n^{m_n}, \sum_{l=1}^n m_l=k.$ 
Since $f_1,\cdots, f_n$ are Poisson central elements, it follows that $I^k/I^{k+1}$ is a free
Poisson $S/I$-module with the basis consisting of monomials $f_1^{m_1}\cdots f_{n}^{m_n}.$
Thus, $I^{k}/I^{k+1}=\oplus_{m_1,\cdots, m_n} (S/I)f_1^{m_1}\cdots f_{n}^{m_n}.$
Let us denote by $\bar{h}$ the image of $h$ in $I^k/I^{k+1}.$
Let us write $\bar{h}=\sum a_{m_1,\cdots, m_n}f_1^{m_1}\cdots f_n^{m_n}, a_{m_1,\cdots, m_n}\in S/I.$
Since $\lbrace S, h \rbrace=0,$ it follows that $a_{m_1,\cdots, m_n}\in (S/I)^p$ (since the Poisson
center of $S/I$ is $(S/I)^p$). Therefore,
 the image of
$h$ in $I^k/I^{k+1}$ must lie in 
$S^p[f_1,\cdots,  f_n]/I^{k+1},$ a contradiction.
\end{proof}

\acknowledgement{ I am very grateful to the anonymous referee for many
helpful suggestions.}

\end{document}